\documentclass[12pt,a4paper]{article}

\usepackage{amsmath,amsfonts,amssymb, amsthm,enumerate,graphicx}
\usepackage{tikz,authblk}
\usepackage{subfig}

\newcommand{\TPSSS}{\mathbb S^{\hspace{.2mm}1} \mbox{$\times
\hspace{-2.8mm}_{-}$} \, \mathbb S^{\hspace{.1mm}3}}

\newtheorem{theorem} {{\textsf{Theorem}}}
\newtheorem{proposition}[theorem]{{\textsf{Proposition}}}
\newtheorem{corollary}[theorem]{{\textsf{Corollary}}}

\newtheorem{definition}{{\textsf{Definition}}}
\newtheorem{remark}{{\textsf{Remark}}}

\newcommand{\e}{\varepsilon}

\usepackage{latexsym}
\usepackage{amsfonts}
\usepackage{amssymb}
\usepackage{amsmath}
\usepackage[english]{babel}

\begin{document}
\title{Lower bounds for regular genus and gem-complexity of PL 4-manifolds}
\author{Biplab Basak$^1$ and Maria Rita Casali$^2$}

\date{}

\maketitle

\vspace{-10mm}
\begin{center}
\noindent {\small $^1$Department of Mathematics, Indian Institute of Science, Bangalore 560\,012, India. biplab10@math.iisc.ernet.in (Current address: Theoretical Statistics and Mathematics Unit, Indian Statistical Institute, Bangalore 560\,059, India.)}

\noindent {\small $^2$Department of Phisics, Mathematics and Computer Science, University of Modena and Reggio Emilia, Via Campi 213 B, I-41125 Modena, Italy. casali@unimore.it}

\medskip

\date{June 27, 2016}
\end{center}

\hrule

\begin{abstract}
Within crystallization theory, two interesting PL invariants for  $d$-manifolds have been introduced and  studied, namely {\it gem-complexity} and  {\it regular genus}. In the present paper we prove that, for any closed connected PL $4$-manifold $M$, its gem-complexity $\mathit{k}(M)$ and its regular genus $ \mathcal G(M)$ satisfy: $$\mathit{k}(M) \ \ge \ 3 \chi (M) + 10m -6 \ \ \  \text{and} \ \ \  \mathcal G(M) \ \ge \ 2 \chi (M) + 5m -4,$$ where $rk(\pi_1(M))=m.$
These lower bounds enable to strictly improve previously known estimations for regular genus and gem-complexity of product 4-manifolds.
Moreover, the class of {\it semi-simple crystallizations} is introduced, so that the represented PL 4-manifolds attain the above lower bounds. The additivity of both gem-complexity and regular genus with respect to connected sum is also proved for such a class of PL 4-manifolds, which comprehends all ones of ``standard type'', involved in existing crystallization catalogues, and their connected sums.

\end{abstract}

\noindent {\small {\em MSC 2010\,:} Primary 57Q15. Secondary 57Q05, 57N13, 05C15.

\noindent {\em Keywords:} PL-manifold, pseudo-triangulation, crystallization, regular genus, gem-complexity, semi-simple crystallization.

}

\medskip

\hrule
\section{Introduction}
\section{Introduction}
A {\em simplicial cell complex} $K$ of dimension $d$ is a poset isomorphic to the face poset ${\mathcal X}$ of a $d$-dimensional simplicial CW-complex $X$. The topological space $X$ is called the {\em geometric carrier} of $K$ and is also denoted by $|K|$. If a topological space $M$ is homeomorphic to $|K|$, then $K$ is said to be a {\em pseudo-triangulation} of $M$.

If a pseudo-triangulation $K$ of a $d$-manifold $M$ ($d \geq 1$) contains exactly $d+1$ vertices, then $K$ is said to be {\em contracted}; its dual graph  gives rise to a {\em crystallization} of $M$, i.e. a $(d+1)$-colored contracted graph $\Gamma = (V, E)$ with an edge coloring $\gamma : E \to \{0, \dots, d\},$ so that the vertices of $K$ have one to one correspondence with the colors $0, \dots, d$ and the facets of $K$ have one to one correspondence with the vertices in $V$ (for details see \cite{fgg86}, or the following
Subsection \ref{crystal}).

The existence of crystallizations for every closed connected PL-manifold is ensured by a classical theorem due to Pezzana (see \cite{pe74}, or \cite{fgg86} for subsequent generalizations). Hence, every closed connected  PL $d$-manifold $M$ admits a contracted pseudo-triangulation, which is also called a {\it colored triangulation} of $M$, because of the edge-coloring of the associated crystallization.

\medskip

Within crystallization theory, two interesting PL invariants for PL $d$-manifolds have been introduced, namely {\it regular genus} and {\it gem-complexity}.

\smallskip

First, let us recall that, if $(\Gamma, \gamma)$ is a $(d+1)$-colored graph,
an embedding $i : \Gamma \hookrightarrow F$ of $\Gamma$ into a closed surface $F$ is called {\it regular} if there exists a cyclic permutation $\varepsilon=(\varepsilon_0, \varepsilon_1, \dots, \varepsilon_d)$ of the color set $\Delta_d=\{0, \dots, d\},$ such that the boundary of each face of $i(\Gamma)$ is a bi-colored cycle with colors $\varepsilon_j, \varepsilon_{j+1}$ for some $j$ (where the addition is performed modulo $d+1$).

Then, we have the following:

\begin{definition}
{\rm
The {\it regular genus} $\rho (\Gamma)$ of $(\Gamma,\gamma)$ is the least genus (resp. half of genus) of an
orientable (resp. non-orientable) surface into which $\Gamma$ embeds regularly; the {\it regular genus} $\mathcal G (M)$ of a closed connected PL $d$-manifold $M$ is defined as the minimum regular genus of its crystallizations.}
\end{definition}

Note that the notion of regular genus extends classical notions to arbitrary dimension: in fact, the regular genus of a closed connected orientable (resp. non-orientable) surface coincides with its genus (resp. half of its genus), while the regular genus of a closed connected 3-manifold coincides with its Heegaard genus (see  \cite{ga81,[GG]}).  The invariant regular genus has been intensively studied, yielding some important general results: for example, regular genus zero characterizes the $d$-sphere among all closed connected PL $d$-manifolds (\cite{[FG$_2$]}). In particular, in dimension $d \in \{4,5\}$, a lot of classifying results in PL-category have been obtained, both for closed and bounded PL $d$-manifolds (via suitable extensions of the involved notions): they concern the case of ``low'' regular genus, the case of ``restricted gap'' between the regular genus of the manifold and the regular genus of its boundary, and the case of ``restricted gap'' between the regular genus and the rank of the fundamental group of the manifold (see, for example, \cite{[CG], [C], {Casali}, [CM]}).

\smallskip
The second definition is quite natural, and is directly related to the combinatorial ``complicatedness'' of the representing tool via crystallization theory:

\begin{definition}\label{defn:gem-complexity}
{\rm
Given a PL $d$-manifold $M$, its {\em gem-complexity} is the non-negative
integer $\mathit{k}(M)=p-1$, where $2p$ is the minimum order of a crystallization of $M$.}
\end{definition}

It is easy to check that, for any dimension $d \ge 2$, gem-complexity zero characterizes the $d$-sphere among all closed connected PL $d$-manifolds.
Moreover, gem-complexity is the natural invariant used to create automatic catalogues of PL-manifolds via crystallizations.
This approach has been successfully followed in dimension three and four, where the choice and implementation of suitable sets of combinatorial moves preserving the represented manifold allowed the development of an effective  ``classifying algorithm'' up to PL-homeomorphism: all $3$-manifolds up to gem-complexity $14$ have been identified (see \cite{[L], [CC$_1$], [CC$_2$]} for the orientable case and \cite{[C$_2$], [C$_4$], [BCrG$_1$]} for the non-orientable one), together with all PL $4$-manifolds up to gem-complexity $8$ (see \cite{Casali14Cataloguing}).

Obviously, any crystallization of a given $d$-manifold $M$ yields an upper bound both for the regular genus and for the gem-complexity of $M$; on the contrary, the problem of finding lower bounds is generally more difficult. In  \cite{bd14}, a lower bound for gem-complexity of a closed connected PL 3-manifold is obtained, by means of the {\it weight} of the fundamental group of $M$, while a lower bound for simply-connected PL 4-manifold, involving the second Betti of $M$, easily follows from \cite[Proposition 2]{[Cav]}\footnote{Actually, \cite[Proposition 2]{[Cav]} yields also a lower bound for  closed connected orientable PL 4-manifold, but in the not simply-connected case it is not so significant.}.

In the present paper, we present lower bounds both for gem-complexity and for regular genus, for the whole class of closed connected PL
4-manifolds:\footnote{From now on, for sake of simplicity, we will simply write ``PL manifold'' instead of ``closed connected PL manifold''.}

\begin{theorem} \label{theorem 0}
Let $M$ be a (closed connected) PL $4$-manifold with $rk(\pi_1(M))=m.$
Then,
$$ \mathit{k}(M) \ \ge \ 3 \chi (M) + 10m -6,$$
$$ \mathcal G(M) \ \ge \ 2 \chi (M) + 5m -4.$$
\end{theorem}

\bigskip

The concept of simple crystallizations was introduced in \cite{bs14}. By definition, PL 4-manifolds represented by simple crystallizations are simply-connected; moreover, the characterization of the class of $4$-manifolds admitting simple crystallizations (see \cite{Casali14GemCompl}) easily proves that all elements of that class attain both bounds of Theorem \ref{theorem 0} (see Remark \ref{rem:simple}).
Hence, in order to investigate the possible sharpness of the above bounds, also in the not simply-connected case, the present paper introduces the concept of {\it semi-simple crystallizations}, which comprehend and generalize simple crystallizations.

\begin{definition}
{\rm A crystallization $(\Gamma,\gamma)$ of a PL $4$-manifold $M$ is called a {\em semi-simple crystallization of type m} if the 1-skeleton of the associated colored triangulation contains exactly $m +1$ 1-simplices for each pair of 0-simplices, where $m$ is the rank of the fundamental group of $M$.
Semi-simple crystallizations of type $0$ are called {\em simple crystallizations}, according to \cite{bs14,Casali14GemCompl}.}
\end{definition}

Note that all PL 4-manifolds involved in the existing crystallization catalogues turn out to admit a semi-simple crystallization (Proposition \ref{prop:gem}); moreover, the class of PL 4-manifolds admitting semi-simple crystallizations is proved to be closed under connected sum (Proposition \ref{prop:connected_sum}). Hence, PL $4$-manifolds admitting semi-simple crystallizations actually constitutes a huge class  (Remark \ref{rem:huge-class}), which comprehends all PL 4-manifolds ``of standard type'' (i.e.  $\mathbb S^4$, $\mathbb{CP}^{2}$, $\mathbb{S}^{2} \times \mathbb{S}^{2}$,  $\mathbb{RP}^4,$  the orientable and non-orientable $\mathbb S^3$-bundles over $\mathbb S^1$ and the $K3$-surface, together with their connected sums, possibly by taking copies with reversed orientation, too).

We prove that each PL 4-manifold in the above class attains both the bounds of Theorem \ref{theorem 0}.

\begin{theorem} \label{theorem 1}
Let $M$ be a  PL $4$-manifold with $rk(\pi_1(M))=m.$
If $M$ admits semi-simple crystallizations, then:
$$ \begin{array}{lll}
\mathit{k}(M) &=& 3 \chi (M) + 10m -6;\\
\mathcal G(M)&=& 2 \chi (M) + 5m -4;\\
\mathit{k}(M) &=& \frac{3 \mathcal G(M)  +5m} 2.
\end{array}
$$
\end{theorem}

We also prove additivity of both gem-complexity and regular genus for the class of  PL $4$-manifolds admitting semi-simple crystallizations.

\begin{theorem} \label{theorem 2}
Let $M_1$ and $M_2$ be two PL $4$-manifolds admitting semi-simple crystallizations.
Then,
 $$ \mathit{k}(M_1 \# M_2) = \mathit{k}(M_1) + \mathit{k}(M_2)   \quad and \quad  \mathcal G(M_1 \# M_2) = \mathcal G(M_1) + \mathcal G(M_2).$$
\end{theorem}

Note that the inequality $ \mathit{k}(M \# M^{\prime}) \leq  \mathit{k}(M) + \mathit{k}(M^{\prime})$ (resp. $\mathcal G(M \# M^{\prime}) \le  \mathcal G(M) + \mathcal G(M^{\prime})$) can be stated for all PL $d$-manifolds by direct estimation of $\mathit{k}(M \# M^{\prime})$ (resp. of $\mathcal G(M \# M^{\prime})$) on
any crystallization $(\Gamma \# \Gamma^{\prime},\gamma \# \gamma^{\prime})$ obtained by graph-connected sum (see Subsection \ref{crystal}),
when $(\Gamma, \gamma)$ and $(\Gamma^{\prime}, \gamma^{\prime})$ are assumed to be  crystallizations of $M$ and $M^{\prime}$ respectively, realizing gem-complexity (resp. regular genus) of the represented $d$-manifolds.
Moreover, we point out that the additivity of  regular genus under connected sum has been conjectured, and the associated (open) problem is significant especially in dimension four. In fact, in dimension four, additivity of regular genus,
at least in the simply-connected case, would imply the 4-dimensional Smooth Poincar\'e Conjecture, in virtue of a well-known Wall's Theorem (\cite{[W]}).

\bigskip

Finally, as an application of Theorem \ref{theorem 0}, we provide lower bounds for regular genus and gem-complexity of product 4-manifolds, which strictly improve previous results (see Section \ref{sec:consequences}).

\section{Preliminaries}  \label{sec:prelim}

\subsection{Basic notions on crystallization theory} \label{crystal}

Crystallization theory is a representation method for the whole class of piecewise linear (PL) manifolds, without restrictions about dimension, connectedness, orientability or boundary properties. In the following brief review, however, we restrict our attention to the closed and connected case, which is the one of interest of the present paper.  We refer to \cite{bm08} for standard terminology on graphs, and to \cite{bj84} for CW-complexes and related notions.

\smallskip

A {\it (d+1)-colored graph} is a pair
$(\Gamma,\gamma)$, where $\Gamma= (V(\Gamma),$ $E(\Gamma))$ is a regular multigraph (i.e. multiple edges are allowed, while loops are forbidden) of degree $d+1$ and $\gamma : E(\Gamma) \to \Delta_d=\{0,1, \dots , d\}$ is a proper edge-coloring (i.e. $\gamma(e) \ne \gamma(f)$ for any pair $e,f$ of adjacent edges).

\smallskip

The elements of the set $\Delta_d$ are called the {\it colors} of
$\Gamma$; moreover, for every $i\in \Delta_d$, an {\it $i$-colored
edge} is an element $e \in E(\Gamma)$ such that $\gamma(e)=i.$

For each $B \subseteq \Delta_d$ with $h$ elements, then the
graph $\Gamma_B =(V(\Gamma), \gamma^{-1}(B))$ is a $h$-colored graph with edge-coloring $\gamma|_{\gamma^{-1}(B)}$.
If $\Gamma_{\Delta_d \setminus\{c\}}$ is connected for all $c\in \Delta_d$, then  $(\Gamma,\gamma)$ is called {\em contracted}.

Let $(\Gamma_1,\gamma_1)$ and $(\Gamma_2,\gamma_2)$ be two disjoint $(d+1)$-colored graphs and let $v_i \in V_i$
for any $i \in \{1,2\}.$  The {\em graph connected sum} of $\Gamma_1$, $\Gamma_2$  with respect to vertices $v_1, v_2$ (denoted by $(\Gamma_1\#_{v_1v_2}\Gamma_2, \gamma_1 \#\gamma_2)$, or simply $(\Gamma_1\#\Gamma_2, \gamma_1 \#\gamma_2)$ when vertices $v_1, v_2$ may be understood)
is the graph obtained from $\Gamma_1$ and $\Gamma_2$ by deleting $v_1$ and $v_2$ and welding the ``hanging'' edges of the same color.

\smallskip

Each $(d+1)$-colored graph uniquely determines a $d$-dimensional simplicial cell-complex ${\mathcal K}(\Gamma)$, which is said to be {\it associated to $\Gamma$}:
\begin{itemize}
\item{} for every vertex $v\in V(\Gamma)$, take a $d$-simplex $\sigma(v)$ and label injectively its $d+1$ vertices by the colors of $\Delta_d$;
\item{} for every $i$-colored edge between $v,w\in V(\Gamma)$, identify the ($d-1$)-faces of $\sigma(v)$ and $\sigma(w)$ opposite to $i$-labelled vertices, so that equally labelled vertices coincide.
\end{itemize}

If the geometrical carrier $|{\mathcal K}(\Gamma)|$ is PL-homeomorphic to a PL $d$-manifold $M$, then the $(d+1)$-colored graph $(\Gamma,\gamma)$ is said to {\it represent} $M$; if, in addition, $(\Gamma,\gamma)$ is contracted, then it is called a {\it crystallization} of $M$.

In both cases,  ${\mathcal K}(\Gamma)$ turns out to be a {\it pseudo-triangulation} of $M$; by taking into account the vertex-labelling inherited from $\gamma$, ${\mathcal K}(\Gamma)$ is also called a {\it colored triangulation} of $M$.

Note that, if $(\Gamma, \gamma)$ is a crystallization of a PL $d$-manifold $M$, then the number of vertices in ${\mathcal K}(\Gamma)$ is $d+1$ (and hence ${\mathcal K}(\Gamma)$ is a {\it contracted pseudo-triangulation} of $M$).  On the other hand, if $K$ is a contracted pseudo-triangulation of $M,$ then the dual graph $\Lambda(K)$ gives rise to a crystallization of $M$.

\smallskip

The following proposition collects some classical results of crystallization theory, which will be useful in the present paper. For details see \cite{fgg86}, together with its references.

\begin{proposition} \label{preliminaries}
Let $(\Gamma,\gamma)$ be a crystallization of a PL $d$-manifold $M$, with $d \geq 3$. Then:
\begin{itemize}
\item[(a)] $M$ is orientable if and only if \ $\Gamma$ is bipartite.
\item[(b)] For each $B \subseteq \Delta_d$ with $h$ elements, there is a bijection between ($d-h$)-simplices of ${\mathcal K}(\Gamma)$ whose vertices are labelled by $\Delta_d \setminus B$  and connected components of $\Gamma_B.$
\item[(c)] For any distinct $r, s \in \Delta_d$ (resp. $i, j, k \in \Delta_d$), let  $g_{rs}$ (resp. $g_{ijk}$) denote the number of connected components of  $\Gamma_{\{r,s\}}$ (resp. $\Gamma_{\{i,j,k\}}$). Then,   $2g_{ijk}=g_{ij}+g_{ik}+g_{jk}-\frac{\#V(\Gamma)}{2}$ for any distinct $i,j,k \in \Delta_d$.
\item[(d)] For any distinct $i, j \in \Delta_d$, the set of connected components of  $\Gamma_{\Delta_d \setminus \{i,j\}}$, but one, is in bijection with a set of generators of the fundamental group $\pi_1(M).$
\item[(e)] If $(\Gamma^{\prime},\gamma^{\prime})$ is a crystallization of a PL $d$-manifold $M^{\prime}$ then, for each $v\in V(\Gamma)$ and $v^\prime\in V(\Gamma^\prime),$ the graph connected sum $(\Gamma \#_{vv^\prime} \Gamma^{\prime}, \gamma \# \gamma^{\prime})$ is a crystallization of a connected sum of $M$ and $M^\prime$. Moreover, if two distinct connected sums of $M$ and $M^\prime$ exist, they both may be represented via graph connected sum  of $\Gamma$ and $\Gamma^{\prime}$, by a suitable choice of $v, v^\prime.$
\end{itemize}
\end{proposition}

\subsection{The regular genus of PL $d$-manifolds}\label{sec:genus}

As already briefly recalled in Section 1, the notion of {\it regular genus} is strictly related to the existence of {\it regular embeddings} of crystallizations into closed surfaces, i.e. embeddings whose regions are bounded by the images of bi-colored cycles, with colors consecutive in a fixed permutation of the color set.

More precisely, according to \cite{ga81},  if $(\Gamma, \gamma)$ is a crystallization of an orientable (resp. non-orientable) PL $d$-manifold $M$ ($d \geq 3$), for each cyclic permutation  $\varepsilon= (\varepsilon_0, \varepsilon_1, \varepsilon_2, \dots , \varepsilon_d)$ of $\Delta_d$,  a regular embedding $i_\varepsilon : \Gamma \hookrightarrow F_\varepsilon$ exists,  where $F_{\e}$ is the closed orientable (resp. non-orientable) surface with Euler characteristic
\begin{eqnarray}\label{relation_chi}
\chi_{\varepsilon}(\Gamma)= \sum_{i \in \mathbb{Z}_{d+1}}g_{\varepsilon_i\varepsilon_{i+1}} + (1-d) \ \frac{\#V(\Gamma)}{2}.
\end{eqnarray}
In the orientable (resp. non-orientable) case, the integer
$$\rho_{\varepsilon}(\Gamma) = 1 - \chi_{\varepsilon}(\Gamma)/2$$
is equal to the genus (resp. half of the genus) of the surface $F_{\varepsilon}$.

Then, by Definition 1, the regular genus $\rho (\Gamma)$ of $(\Gamma,\gamma)$  and the regular genus $\mathcal G (M)$ of $M$ are:
$$\rho(\Gamma)= \min \{\rho_{\varepsilon}(\Gamma) \ | \  \varepsilon \ \text{ is a cyclic permutation of } \ \Delta_d\};$$
$$\mathcal G(M) = \min \{\rho(\Gamma) \ | \  (\Gamma,\gamma) \mbox{ is a crystallization of } M\}.$$

\medskip

Note that $\mathcal G(M) \geq rk (\pi_1(M))$ is known to hold, for any PL $d$-manifold ($d\geq 3$). In the 4-dimensional settings, the following results about the ``gap'' between the regular genus and the rank of the fundamental group of a PL 4-manifold have been obtained (see \cite{[C], [CM]}).

\begin{proposition} \label{gap_genus-rank}
Let $M$ be a PL $4$-manifold. Then:
\begin{itemize}
\item[(a)] If \ $\mathcal G(M) = rk (\pi_1(M))= \rho,$ \ then $M$ is PL-homeomorphic to \ $\#_{\rho} (\mathbb S^1 \otimes \mathbb S^3)$,  \ where $\mathbb S^1 \otimes \mathbb S^3$ denotes either the orientable or non-orientable $\mathbb S^3$-bundle over $\mathbb S^1$, according to the orientability of $M$.
\item[(b)] No PL $4$-manifold $M$ exists with \ $\mathcal G(M) = rk (\pi_1(M)) +1$.
\item[(c)] If \ $\mathcal G(M) = rk (\pi_1(M)) +2 $ and $\pi_1(M)= \ast_m \mathbb{Z}$, \ then $M$ is PL-homeomorphic to \  $\mathbb{CP}^2\#_m (\mathbb S^1 \otimes \mathbb S^3)$.
 \end{itemize}
\end{proposition}

Moreover, if $(\Gamma, \gamma)$ is a crystallization of a PL $4$-manifold and $\varepsilon = (\varepsilon_0,\dots, \varepsilon_4)$ is a cyclic permutation of the color set $\Delta_4$, then
 \cite[relations (1$_j$) and (2$_j$)]{[CM]} yields:
$$g_{(i-1)(i+1)}=g_{(i-1)(i)(i+1)} + \rho - \rho_{\hat i} \ \ \  \quad \text {and}  \quad \ \ \
g_{(i-1)(i+1)(i+2)}= 1 + \rho - \rho_{\hat i} - \rho_{\hat {i+3}},$$
where $\rho$ denotes $\rho_{\varepsilon}(\Gamma)$ and, for any $i \in \mathbb Z_5,$  $\rho_{\hat i}$ denotes $\rho_{\varepsilon} (\Gamma_{\Delta_4 \setminus \{i\}}).$

\section{Lower bounds for regular genus and gem-complexity in dimension 4}

Let us now prove the general result (Theorem \ref{theorem 0}, already stated in Section 1) yielding lower bounds for gem-complexity and regular genus of any PL 4-manifold. In Section \ref{sec:semi-simple} it will be of fundamental importance in order to analyze the properties of PL 4-manifolds admitting semi-simple crystallizations, while in Section \ref{sec:consequences} it will enable to obtain new estimations for both invariants in the case of product 4-manifolds. On the other hand we believe that, thanks to its generality, it could be useful to investigate PL 4-manifolds also in wider contexts.

\bigskip

\noindent {\em Proof of Theorem} \ref{theorem 0}.
Let $(\Gamma,\gamma)$ be a crystallization of $M$. If  $2p=\#V(\Gamma),$ then $X=\mathcal{K}(\Gamma)$ is a $2p$-facet contracted pseudo-triangulation of the PL 4-manifold $M$.
The Dehn-Sommerville equations in dimension four yield:
	$$
	\begin{array}{lllllllllll}
		f_0 (X) &-& f_1(X) &+& f_2(X) &-& f_3(X) &+& f_4(X) &=& \chi(M), \\
		&& 2 f_1(X) &-& 3 f_2(X) &+& 4 f_3(X) &-& 5 f_4(X) &=& 0, \\
		&&&&&& 2 f_3(X) &-& 5 f_4(X) &=& 0. \\
	\end{array}
	$$
\noindent Since $f_0 (X) = 5$ by construction and $f_4(X)=\#V(\Gamma)=2p$, the following equality holds:
\begin{eqnarray}\label{ch6:eq1}
2p = 6 \chi(M)+2f_1(\mathcal{K}(\Gamma))-30.
\end{eqnarray}
\noindent Since $rk(\pi_1(M))=m$, Proposition \ref{preliminaries}(d) implies
$g_{ijk} \geq m+1$ for any distinct $i,j,k \in \Delta_4$. Therefore, $f_1(\mathcal{K}(\Gamma)) = \sum_{0 \leq i <j<k\leq 4}g_{ijk} \geq 10(m+1)$. Hence, by equation \eqref{ch6:eq1}:
\begin{eqnarray}\label{ch6:eq2}
 2p \geq 6 \chi(M)+ 20 (m+1) -30 \ = \ 6 \chi(M) + 20 m - 10.
\end{eqnarray}
\noindent The first inequality of Theorem 1 now follows from equation \eqref{ch6:eq2} and Definition \ref{defn:gem-complexity}.

\medskip

Let us now prove the second inequality.
From equation \eqref{ch6:eq2}, we have that $ 2\bar{p}=6 \chi(M)+10(2m-1)$ is the minimal possible order of a crystallization of $M.$

Let  $(\Gamma,\gamma)$ be a crystallization of $M$. Then, $\#V(\Gamma)=2\bar p+2q$ for
some non-negative integer $q$. This implies $6\chi(M)+2f_1(\mathcal{K}(\Gamma))-30 = 6 \chi(M)+10(2m-1) +2q$.
Thus, $2\sum_{0 \leq i <j<k\leq 4}$ $g_{ijk}-30 = 10(2m-1) +2q$. Again, $g_{ijk} \geq m+1$ for any distinct $i,j,k \in \Delta_4$. So, let us assume
$g_{ijk}=(m+1)+t_{ijk}$ where $t_{ijk} \in \mathbb{Z}$, $t_{ijk} \geq 0$.
Thus, $20(m+1)+ 2\sum_{0 \leq i <j<k\leq 4}t_{ijk}-30 = 20m-10 +2q$ and hence
$q=\sum_{0 \leq i <j<k\leq 4}t_{ijk}$. Now, for any cyclic permutation
$\varepsilon = (\varepsilon_0,\varepsilon_1, \dots, \varepsilon_4)$
of the colors we have $\chi_{\varepsilon}(\Gamma)=\sum_{i \in \mathbb{Z}_5}g_{\varepsilon_i\varepsilon_{i+1}}-3(\bar p+q)$
(by relation \eqref{relation_chi} of the previous Section).

On the other hand, by Proposition \ref{preliminaries}(c), we know that $2g_{ijk}=g_{ij}+g_{ik}+g_{jk}-\frac{\#V(\Gamma)}{2}$ for any distinct $i,j,k \in \Delta_4$.
Thus, we have $g_{ij}+g_{ik}+g_{jk}=2g_{ijk}+(\bar p+q)$ for $0 \leq i <j<k\leq 4$. This gives ten linear equations which can be written in the following form:
$$ AX=B,$$
\noindent where
$$A = \begin{bmatrix}
1 & 1 & 0 & 0 & 1 & 0 & 0 & 0 & 0 & 0\\
1 & 0 & 1 & 0 & 0 & 1 & 0 & 0 & 0 & 0\\
1 & 0 & 0 & 1 & 0 & 0 & 1 & 0 & 0 & 0\\
0 & 1 & 1 & 0 & 0 & 0 & 0 & 1 & 0 & 0\\
0 & 1 & 0 & 1 & 0 & 0 & 0 & 0 & 1 & 0\\
0 & 0 & 1 & 1 & 0 & 0 & 0 & 0 & 0 & 1\\
0 & 0 & 0 & 0 & 1 & 1 & 0 & 1 & 0 & 0\\
0 & 0 & 0 & 0 & 1 & 0 & 1 & 0 & 1 & 0\\
0 & 0 & 0 & 0 & 0 & 1 & 1 & 0 & 0 & 1\\
0 & 0 & 0 & 0 & 0 & 0 & 0 & 1 & 1 & 1\\
	\end{bmatrix}, \, X=
	\begin{bmatrix}
	g_{01}\\
	g_{02}\\
	g_{03}\\
	g_{04}\\
	g_{12}\\
	g_{13}\\
	g_{14}\\
	g_{23}\\
	g_{24}\\
	g_{34}\\
	\end{bmatrix} \mbox{and } B=
	\begin{bmatrix}
	2g_{012}+ \bar p+q\\
	2g_{013}+ \bar p+q\\
	2g_{014}+ \bar p+q\\
	2g_{023}+ \bar p+q\\
	2g_{024}+ \bar p+q\\
	2g_{034}+ \bar p+q\\
	2g_{123}+ \bar p+q\\
	2g_{124}+ \bar p+q\\
	2g_{134}+ \bar p+q\\
	2g_{234}+ \bar p+q\\
	\end{bmatrix}.
	$$
\noindent Therefore
$$ X= A^{-1}B,$$
\noindent where
 $$A^{-1}=
 \begin{bmatrix}
1/3 & 1/3 & 1/3 & -1/6 & -1/6 & -1/6 & -1/6 & -1/6 & -1/6 & 1/3\\
1/3 & -1/6 & -1/6 & 1/3 & 1/3 & -1/6 & -1/6 & -1/6 & 1/3 & -1/6\\
-1/6 & 1/3 & -1/6 & 1/3 & -1/6 & 1/3 & -1/6 & 1/3 & -1/6 & -1/6\\
-1/6 & -1/6 & 1/3 & -1/6 & 1/3 & 1/3 & 1/3 & -1/6 & -1/6 & -1/6\\
1/3 & -1/6 & -1/6 & -1/6 & -1/6 & 1/3 & 1/3 & 1/3 & -1/6 & -1/6\\
-1/6 & 1/3 & -1/6 & -1/6 & 1/3 & -1/6 & 1/3 & -1/6 & 1/3 & -1/6\\
-1/6 & -1/6 & 1/3 & 1/3 & -1/6 & -1/6 & -1/6 & 1/3 & 1/3 & -1/6\\
-1/6 & -1/6 & 1/3 & 1/3 & -1/6 & -1/6 & 1/3 & -1/6 & -1/6 & 1/3\\
-1/6 & 1/3 & -1/6 & -1/6 & 1/3 & -1/6 & -1/6 & 1/3 & -1/6 & 1/3\\
1/3 & -1/6 & -1/6 & -1/6 & -1/6 & 1/3 & -1/6 & -1/6 & 1/3 & 1/3\\
	\end{bmatrix}.$$	
\noindent Since
$g_{ijk}=(m+1)+t_{ijk}$ for any distinct $i,j,k \in \Delta_4$, we have
$$B = M+\sum_{0 \leq i <j<k\leq 4} T_{ijk},$$	
\noindent where
$$M= \begin{bmatrix}
	2(m+1)+ \bar p+q\\
	2(m+1)+ \bar p+q\\
	2(m+1)+ \bar p+q\\
	2(m+1)+ \bar p+q\\
	2(m+1)+ \bar p+q\\
	2(m+1)+ \bar p+q\\
	2(m+1)+ \bar p+q\\
	2(m+1)+ \bar p+q\\
	2(m+1)+ \bar p+q\\
	2(m+1)+ \bar p+q\\
	\end{bmatrix} \mbox{ and }
	T_{ijk}= \begin{bmatrix}
	0\\
	0\\
	0\\
	0\\
	2t_{ijk}\\
	0\\
	0\\
	0\\
	0\\
	0\\
	\end{bmatrix}.$$
\noindent Thus,
$$X = A^{-1}M+\sum_{0 \leq i<j<k\leq 4} A^{-1} T_{ijk}.$$
Therefore,
$g_{\varepsilon_i\varepsilon_{i+1}}=\frac{2(m+1)+\bar p+q}{3} + 2\sum_{0 \leq l <j<k\leq 4}c_{ljk}^{\varepsilon_i\varepsilon_{i+1}}t_{ljk}$, where $c_{ljk}^{rs}$ is the element of $A^{-1}$ corresponding to
$\{r,s\}$-row and $\{l,j,k\}$-column of $A^{-1}$.
Now observe that, for any fixed $\{l,j,k\}$-column of $A^{-1}$, there exist $r,s,v$
such that $c_{ljk}^{rs}=c_{ljk}^{rv}=c_{ljk}^{sv}=1/3$. Thus, for any cyclic permutation
$\varepsilon = (\varepsilon_0,\varepsilon_1, \dots, \varepsilon_4)$, at most three elements from the set $\{c_{ljk}^{\varepsilon_i\varepsilon_{i+1}} \ | \ i \in \mathbb{Z}_5\}$ are $1/3$ and remaining elements are $-1/6$.
Therefore,  $\sum_{i \in \mathbb{Z}_5} c_{ljk}^{\varepsilon_i\varepsilon_{i+1}} \leq 1/3+1/3+1/3-1/6-1/6=2/3$. Thus,
\begin{align*}
\chi_{\varepsilon}(\Gamma)& = \sum_{i \in \mathbb{Z}_5}g_{\varepsilon_i\varepsilon_{i+1}}-3(\bar p+q)\\
& =\sum_{i \in \mathbb{Z}_5}\Big(\frac{2(m+1)+ \bar p+ q}{3} + 2\sum_{0 \leq l <j<k\leq 4}c_{ljk}^{\varepsilon_i\varepsilon_{i+1}}t_{ljk}\Big)-3(\bar p+q)\\
&= 5 \, \frac{2(m+1)+ \bar p+q}{3} + 2\sum_{0 \leq l <j<k\leq 4}t_{ljk}
\sum_{i \in \mathbb{Z}_5}c_{ljk}^{\varepsilon_i\varepsilon_{i+1}}-3(\bar p+q)\\
& \leq  5 \, \frac{2(m+1)+ \bar p+q}{3}+ \frac 4 3\sum_{0 \leq l <j<k\leq 4}t_{ljk}-3(\bar p+q)\\
& =  5 \, \frac{2(m+1)+ \bar p+ q}{3}+ \frac 4 3 \, q -3(\bar p+q)\\
& =  \frac{10(m+1)-4 \bar p}{3}.
\end{align*}
\noindent Therefore, $\rho_{\varepsilon}(\Gamma) = 1 - \chi_{\varepsilon}(\Gamma)/2 \geq 1-\frac{5(m+1)-2 \bar p}{3}=
\frac{2(\bar p-1)-5m}{3}$. Since this is true for any cyclic permutation $\varepsilon$, we have:
$$\rho(\Gamma)=\min\{\rho_{\varepsilon}(\Gamma)\, | \, \varepsilon \ \mbox{is a cyclic  permutation of } \Delta_4\} \geq \frac{2(\bar p-1)-5m}{3}.$$
Since the crystallization $(\Gamma,\gamma)$ is arbitrary:
$$\mathcal G(M)=\min\{\rho(\Gamma) \, | \, \Gamma \ \mbox{is a crystallization of }  M\}  \geq \frac{2(\bar p-1)-5m} 3.$$
Finally, since $2\bar p = 6 \chi(M) + 10(2m - 1)$, we have:
$$\mathcal G(M) \geq \frac{(6 \chi(M) + 20 m - 12) - 5m}{3} =  2 \chi (M) + 5m -4.$$
\hfill $\Box$

\begin{remark} \label{compare GG} {\rm As a consequence of the inequality involving regular genus in Theorem \ref{theorem 0}, we have
$$ \chi(M)  \le  2 + \frac {\mathcal G(M)} 2 -  \frac {5 m} 2,$$
which improves the  inequality  $ \chi(M)  \le  2 + {\mathcal G(M)}/2 $  in \cite[Corollary 6.5]{[GG]}. On the other hand, additivity of regular genus is proved in \cite[Corollary 6.8 (b)]{[GG]} for the class of PL 4-manifolds characterized by $\mathcal G(M) \ = \ 2 \chi (M) - 4$. Now, in virtue of Theorem \ref{theorem 0}, the above class of PL 4-manifolds turns out to consist of simply-connected PL 4-manifolds.}
\end{remark}

\begin{remark} \label{rem:simple} {\rm According to \cite{Casali14GemCompl}, (simply-connected) PL 4-manifolds admitting simple crystallizations are characterized by $k(M)= 3 \beta_2(M);$ moreover, equality $\mathcal G(M)= 2 \beta_2(M)$ holds for any PL 4-manifold admitting simple crystallizations. Hence, they constitute a class of (simply-connected) PL 4-manifolds which attain both the bounds of Theorem \ref{theorem 0}.}
\end{remark}

\section{PL 4-manifolds admitting semi-simple crystallizations}
\label{sec:semi-simple}

In Section 1, the notion of {\it semi-simple crystallization} has been introduced, in terms of the 1-skeleton of the associated colored triangulation.
By Proposition \ref{preliminaries}(b), it is easy to check that  Definition 3 may be re-stated as follows, in terms of the number of components of the subgraph restricted to any triple of colors.

\begin{definition}
{ \rm A crystallization $(\Gamma,\gamma)$ of a PL $4$-manifold $M$ is called a {\em semi-simple crystallization of type $m$} if $g_{ijk} = m+1$ for any distinct $i,j,k \in \Delta_4$, where $m$ is the rank of the fundamental group of $M$.}
\end{definition}

The following definition is quite natural.

\begin{definition}
{\rm If $(\Gamma,\gamma)$ is a  semi-simple crystallization of type $m$ of a PL $4$-manifold $M$, then
we will say that $M$ {\it admits semi-simple crystallizations of type $m$} (or simply, that {\it $M$ admits semi-simple crystallizations}).}
\end{definition}

\begin{remark}\label{rem:characterization}
{\rm
The first part of the proof of Theorem \ref{theorem 0} (in particular, equation \eqref{ch6:eq1}) immediately implies that $(\Gamma,\gamma)$ is a semi-simple crystallization of type $m$ (that is, $f_1(\mathcal{K}(\Gamma))=10(m+1)$) if and only if $\#V(\Gamma) = 6 \chi(M)+20m-10.$  Hence, {\it PL $4$-manifolds admitting semi-simple crystallizations are characterized by} $\mathit{k}(M)= 3\chi(M)+10m-6$, where $m$ is the rank of the fundamental group of $M$.
}
\end{remark}

\begin{proposition} \label{prop:connected_sum}
Let $M$ and $M^{\prime}$ be two PL $4$-manifolds admitting semi-simple crystallizations.
Then,  $M \# M^{\prime}$ admits semi-simple crystallizations, too.
\end{proposition}

\begin{proof}
Let  $(\Gamma,\gamma)$ (resp. $(\Gamma^{\prime},\gamma^{\prime})$) be a semi-simple crystallization  of $M$ (resp. $M^{\prime}$), with $g_{ijk} = m+1$ (resp. $g_{ijk}^{\prime} = m^{\prime} +1$) for any distinct $i,j,k \in \Delta_4$.
Let $(\bar \Gamma,\bar\gamma)$ be a crystallization of $M \# M^{\prime}$  obtained by graph connected sum of $(\Gamma,\gamma)$ and $(\Gamma^{\prime},\gamma^{\prime})$ (see Subsection \ref{crystal}).
By construction, $(\bar \Gamma,\bar\gamma)$ has $\bar g_{ijk} = m + m^{\prime} +1$ for all distinct $i,j,k \in \Delta_4$. The thesis now easily follows.
\end{proof}

\begin{figure}[ht]
\tikzstyle{vert}=[circle, draw, fill=black!100, inner sep=0pt, minimum width=4pt] \tikzstyle{vertex}=[circle, draw, fill=black!00, inner sep=0pt, minimum width=4pt] \tikzstyle{ver}=[]
\tikzstyle{extra}=[circle, draw, fill=black!50, inner sep=0pt, minimum width=2pt] \tikzstyle{edge} = [draw,thick,-] \tikzstyle{arrow} = [draw,thick,->]
\centering
\begin{tikzpicture}[scale=0.3]
\begin{scope}[shift={(-12,0)}]
\foreach \x/\y/\z in {-8/-4/v1,-4/4/v4,0/-4/v5,4/4/v8,8/-4/v9}{
 \node[vert] (\z) at (\y,\x){};}

\foreach \x/\y/\z in {-8/4/v2,-4/-4/v3,0/4/v6,4/-4/v7,8/4/v10}{
 \node[vertex] (\z) at (\y,\x){};}

 \path[edge] (v1) -- (v3);
\path[edge] (v3) -- (v5);
\draw [line width=2pt, line cap=rectengle, dash pattern=on 1pt off 1]  (v1) -- (v3);
\draw [line width=3pt, line cap=round, dash pattern=on 0pt off 2\pgflinewidth]  (v3) -- (v5);
\path[edge] (v5) -- (v7);
\path[edge, dotted] (v7) -- (v9);

 \path[edge] (v2) -- (v4);
\path[edge] (v4) -- (v6);
\draw [line width=2pt, line cap=rectengle, dash pattern=on 1pt off 1]  (v2) -- (v4);
\draw [line width=3pt, line cap=round, dash pattern=on 0pt off 2\pgflinewidth]  (v4) -- (v6);
\path[edge] (v6) -- (v8);
\path[edge, dotted] (v8) -- (v10);

 \path[edge] (v9) -- (v10);
\path[edge] (v1) -- (v2);
\draw [line width=2pt, line cap=rectengle, dash pattern=on 1pt off 1]  (v9) -- (v10);
\draw [line width=3pt, line cap=round, dash pattern=on 0pt off 2\pgflinewidth]  (v1) -- (v2);
\path[edge] (v3) -- (v4);
\path[edge, dotted] (v5) -- (v6);
\path[edge, dashed] (v7) -- (v8);

\draw[edge] plot [smooth,tension=1.5] coordinates{(v1)(0,-9)(v2)};
\draw[edge, dotted] plot [smooth,tension=1.5] coordinates{(v1)(0,-7)(v2)};
\draw[edge,dotted] plot [smooth,tension=1.5] coordinates{(v3)(0,-5)(v4)};
\draw[edge, dashed] plot [smooth,tension=1.5] coordinates{(v3)(0,-3)(v4)};
\draw[edge,dashed] plot [smooth,tension=1.5] coordinates{(v5)(0,-1)(v6)};
\draw[edge] plot [smooth,tension=1.5] coordinates{(v5)(0,1)(v6)};
\draw[line width=2pt, line cap=rectengle, dash pattern=on 1pt off 1] plot [smooth,tension=1.5] coordinates{(v5)(0,1)(v6)};
\draw[edge] plot [smooth,tension=1.5] coordinates{(v7)(0,3)(v8)};
\draw[line width=2pt, line cap=rectengle, dash pattern=on 1pt off 1] plot [smooth,tension=1.5] coordinates{(v7)(0,3)(v8)};
\draw[edge] plot [smooth,tension=1.5] coordinates{(v7)(0,5)(v8)};
\draw[line width=3pt, line cap=round, dash pattern=on 0pt off 2\pgflinewidth] plot [smooth,tension=1.5] coordinates{(v7)(0,5)(v8)};
\draw[edge] plot [smooth,tension=1.5] coordinates{(v9)(0,7)(v10)};
\draw[line width=3pt, line cap=round, dash pattern=on 0pt off 2\pgflinewidth] plot [smooth,tension=1.5] coordinates{(v9)(0,7)(v10)};
\draw[edge] plot [smooth,tension=1.5] coordinates{(v9)(0,9)(v10)};
\draw[edge, dashed] plot [smooth,tension=1] coordinates{(v1)(-5,-2)(-4,9)(v10)};
\draw[edge, dashed] plot [smooth,tension=1] coordinates{(v2)(5,-2)(4,9)(v9)};
 \end{scope}

\begin{scope}[shift={(12,0)}]
\foreach \x/\y/\z in {-8/-4/v1,-4/4/v4,0/-4/v5,4/4/v8,8/-4/v9}{
 \node[vert] (\z) at (\y,\x){};}

\foreach \x/\y/\z in {-8/4/v2,-4/-4/v3,0/4/v6,4/-4/v7,8/4/v10}{
 \node[vertex] (\z) at (\y,\x){};}

 \path[edge] (v1) -- (v3);
\path[edge] (v3) -- (v5);
\draw [line width=2pt, line cap=rectengle, dash pattern=on 1pt off 1]  (v1) -- (v3);
\draw [line width=3pt, line cap=round, dash pattern=on 0pt off 2\pgflinewidth]  (v3) -- (v5);
\path[edge] (v5) -- (v7);
\path[edge, dotted] (v7) -- (v9);

 \path[edge] (v2) -- (v4);
\path[edge] (v4) -- (v6);
\draw [line width=2pt, line cap=rectengle, dash pattern=on 1pt off 1]  (v2) -- (v4);
\draw [line width=3pt, line cap=round, dash pattern=on 0pt off 2\pgflinewidth]  (v4) -- (v6);
\path[edge] (v6) -- (v8);
\path[edge, dotted] (v8) -- (v10);

 \path[edge] (v9) -- (v10);
\path[edge] (v1) -- (v2);
\draw [line width=2pt, line cap=rectengle, dash pattern=on 1pt off 1]  (v9) -- (v10);
\draw [line width=3pt, line cap=round, dash pattern=on 0pt off 2\pgflinewidth]  (v1) -- (v2);
\path[edge] (v3) -- (v4);
\path[edge, dotted] (v5) -- (v6);
\path[edge, dashed] (v7) -- (v8);

\draw[edge] plot [smooth,tension=1.5] coordinates{(v1)(0,-9)(v2)};
\draw[edge, dotted] plot [smooth,tension=1.5] coordinates{(v1)(0,-7)(v2)};
\draw[edge,dotted] plot [smooth,tension=1.5] coordinates{(v3)(0,-5)(v4)};
\draw[edge, dashed] plot [smooth,tension=1.5] coordinates{(v3)(0,-3)(v4)};
\draw[edge,dashed] plot [smooth,tension=1.5] coordinates{(v5)(0,-1)(v6)};
\draw[edge] plot [smooth,tension=1.5] coordinates{(v5)(0,1)(v6)};
\draw[line width=2pt, line cap=rectengle, dash pattern=on 1pt off 1] plot [smooth,tension=1.5] coordinates{(v5)(0,1)(v6)};
\draw[edge] plot [smooth,tension=1.5] coordinates{(v7)(0,3)(v8)};
\draw[line width=2pt, line cap=rectengle, dash pattern=on 1pt off 1] plot [smooth,tension=1.5] coordinates{(v7)(0,3)(v8)};
\draw[edge] plot [smooth,tension=1.5] coordinates{(v7)(0,5)(v8)};
\draw[line width=3pt, line cap=round, dash pattern=on 0pt off 2\pgflinewidth] plot [smooth,tension=1.5] coordinates{(v7)(0,5)(v8)};
\draw[edge] plot [smooth,tension=1.5] coordinates{(v9)(0,7)(v10)};
\draw[line width=3pt, line cap=round, dash pattern=on 0pt off 2\pgflinewidth] plot [smooth,tension=1.5] coordinates{(v9)(0,7)(v10)};
\draw[edge] plot [smooth,tension=1.5] coordinates{(v9)(0,9)(v10)};
\draw[edge, dashed] plot [smooth,tension=1.5] coordinates{(v1)(-5,0)(v9)};
\draw[edge, dashed] plot [smooth,tension=1.5] coordinates{(v2)(5,0)(v10)};
 \end{scope}

 \begin{scope}[shift={(0,-1)}]
\node[ver] (308) at (-3,5){$0$};
\node[ver] (300) at (-3,3){$1$};
\node[ver] (301) at (-3,1){$2$};
\node[ver] (302) at (-3,-1){$3$};
\node[ver] (303) at (-3,-3){$4$};
\node[ver] (309) at (3,5){};
\node[ver] (304) at (3,3){};
\node[ver] (305) at (3,1){};
\node[ver] (306) at (3,-1){};
\node[ver] (307) at (3,-3){};
\path[edge] (300) -- (304);
\path[edge] (308) -- (309);
\draw [line width=2pt, line cap=rectengle, dash pattern=on 1pt off 1]  (308) -- (309);
\draw [line width=3pt, line cap=round, dash pattern=on 0pt off 2\pgflinewidth]  (300) -- (304);
\path[edge] (301) -- (305);
\path[edge, dotted] (302) -- (306);
\path[edge, dashed] (303) -- (307);
\end{scope}
 \end{tikzpicture}
\caption{Semi-simple crystallizations of $\mathbb S^1 \times \mathbb S^3$ and $\TPSSS$}\label{fig:S2S1}\end{figure}
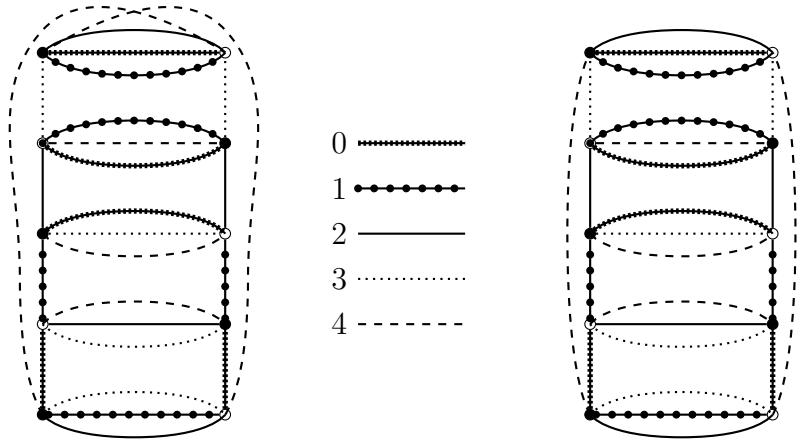

\begin{proposition} \label{prop:gem}
Let $M$ be a PL $4$-manifold with gem-complexity less than nine. Then, $M$ admits semi-simple crystallizations.
\end{proposition}

\begin{proof}
By \cite[Proposition 15]{Casali14Cataloguing}, the  PL $4$-manifolds with gem-complexity less than nine are:
$\mathbb S^4$, $\mathbb{CP}^2$, the orientable $\mathbb S^3$-bundle $\mathbb S^1 \times \mathbb S^3,$  the non-orientable $\mathbb S^3$-bundle $\TPSSS$ and the 4-dimensional real projective space $\mathbb{RP}^4$, together with some suitable connected sums of them.
Both the simply-connected PL $4$-manifolds $\mathbb S^4$ and $\mathbb{CP}^2$ admit simple crystallizations (see \cite{bs14}), i.e. semi-simple crystallizations of type 0.  Moreover, it is easy to check that the (well-known) crystallizations of $\mathbb S^1 \times \mathbb S^3$ and  $\TPSSS$ depicted in Figure \ref{fig:S2S1} are semi-simple crystallizations, as well as the unique order 16 crystallization of $\mathbb{RP}^4$ depicted in Figure \ref{fig:RP4} (see \cite{Casali14Cataloguing} for the uniqueness).
Now the result follows from the additivity of semi-simple crystallizations (Proposition \ref{prop:connected_sum}).
\end{proof}
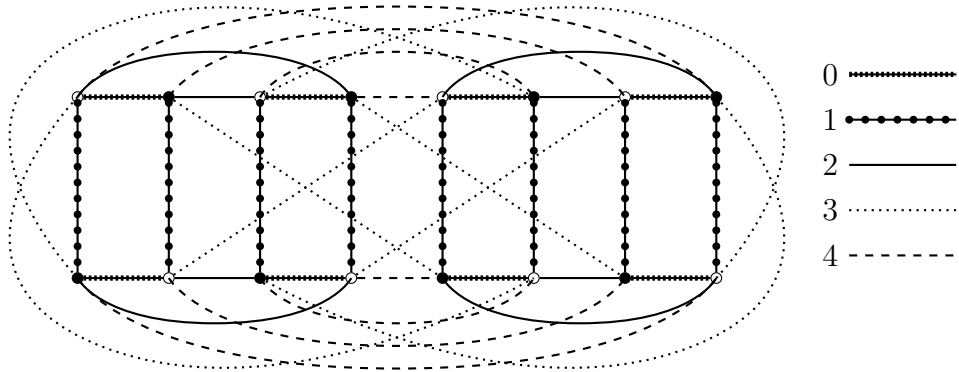
\begin{figure}[ht]
\tikzstyle{vert}=[circle, draw, fill=black!100, inner sep=0pt, minimum width=4pt] \tikzstyle{vertex}=[circle, draw, fill=black!00, inner sep=0pt, minimum width=4pt] \tikzstyle{ver}=[]
\tikzstyle{extra}=[circle, draw, fill=black!50, inner sep=0pt, minimum width=2pt] \tikzstyle{edge} = [draw,thick,-] \tikzstyle{arrow} = [draw,thick,->]
\centering
\begin{tikzpicture}[scale=0.3]
\begin{scope}[]
\foreach \x/\y/\z in {-14/4/v1,-10/-4/v10,-6/4/v3,-2/-4/v12,2/4/v5,6/-4/v14,10/4/v7,14/-4/v16}{
 \node[vertex] (\z) at (\x,\y){};}

\foreach \x/\y/\z in {-14/-4/v9,-10/4/v2,-6/-4/v11,-2/4/v4,2/-4/v13,6/4/v6,10/-4/v15,14/4/v8}{
 \node[vert] (\z) at (\x,\y){};}

\foreach \x/\y in {v1/v2,v3/v4,v5/v6,v7/v8,v9/v10,v11/v12,v13/v14,v15/v16}{
\path[edge] (\x) -- (\y);
\draw [line width=2pt, line cap=rectengle, dash pattern=on 1pt off 1]  (\x) -- (\y);}

\foreach \x/\y in {v1/v9,v2/v10,v3/v11,v4/v12,v5/v13,v6/v14,v7/v15,v8/v16}{
\path[edge] (\x) -- (\y);
\draw [line width=3pt, line cap=round, dash pattern=on 0pt off 2\pgflinewidth]  (\x) -- (\y);}

\foreach \x/\y in {v2/v3,v6/v7,v10/v11,v14/v15}{
\path[edge] (\x) -- (\y);}
\draw[edge] plot [smooth,tension=1.5] coordinates{(v1)(-8,6)(v4)};
\draw[edge] plot [smooth,tension=1.5] coordinates{(v9)(-8,-6)(v12)};
\draw[edge] plot [smooth,tension=1.5] coordinates{(v5)(8,6)(v8)};
\draw[edge] plot [smooth,tension=1.5] coordinates{(v13)(8,-6)(v16)};

\foreach \x/\y in {v4/v5,v12/v13}{
\path[edge, dashed] (\x) -- (\y);}

\draw[edge, dashed] plot [smooth,tension=1.5] coordinates{(v3)(0,6)(v6)};
\draw[edge, dashed] plot [smooth,tension=1.5] coordinates{(v2)(0,7)(v7)};
\draw[edge, dashed] plot [smooth,tension=1.5] coordinates{(v1)(0,8)(v8)};
\draw[edge, dashed] plot [smooth,tension=1.5] coordinates{(v11)(0,-6)(v14)};
\draw[edge, dashed] plot [smooth,tension=1.5] coordinates{(v10)(0,-7)(v15)};
\draw[edge, dashed] plot [smooth,tension=1.5] coordinates{(v9)(0,-8)(v16)};

\foreach \x/\y in {v4/v15,v12/v7,v2/v13,v10/v5}{
\path[edge, dotted] (\x) -- (\y);}
\draw[edge, dotted] plot [smooth,tension=1.5] coordinates{(v1)(-13,-7)(v14)};
\draw[edge, dotted] plot [smooth,tension=1.5] coordinates{(v9)(-13,7)(v6)};
\draw[edge, dotted] plot [smooth,tension=1.5] coordinates{(v8)(13,-7)(v11)};
\draw[edge, dotted] plot [smooth,tension=1.5] coordinates{(v16)(13,7)(v3)};
 \end{scope}

 \begin{scope}[shift={(22,0)}]
\node[ver] (308) at (-3,5){$0$};
\node[ver] (300) at (-3,3){$1$};
\node[ver] (301) at (-3,1){$2$};
\node[ver] (302) at (-3,-1){$3$};
\node[ver] (303) at (-3,-3){$4$};
\node[ver] (309) at (3,5){};
\node[ver] (304) at (3,3){};
\node[ver] (305) at (3,1){};
\node[ver] (306) at (3,-1){};
\node[ver] (307) at (3,-3){};
\path[edge] (300) -- (304);
\path[edge] (308) -- (309);
\draw [line width=2pt, line cap=rectengle, dash pattern=on 1pt off 1]  (308) -- (309);
\draw [line width=3pt, line cap=round, dash pattern=on 0pt off 2\pgflinewidth]  (300) -- (304);
\path[edge] (301) -- (305);
\path[edge, dotted] (302) -- (306);
\path[edge, dashed] (303) -- (307);
\end{scope}

\end{tikzpicture}
 \caption{A semi-simple crystallization of $\mathbb{RP}^4$} \label{fig:RP4}
 \end{figure}

\begin{remark}\label{rem:huge-class}
{\rm Note that (semi-)simple crystallizations of all simply-connected PL $4$-manifolds of ``standard type'' are known (see \cite{bs14}, where simple crystallizations of  $\mathbb S^4$, $\mathbb{CP}^{2}$, $\mathbb{S}^{2} \times \mathbb{S}^{2}$ and the $K3$-surface are presented); moreover, Proposition \ref{prop:gem} yields semi-simple crystallizations of the non-simply-connected PL $4$-manifolds $\mathbb S^1 \times \mathbb S^3$, $\TPSSS$ and $\mathbb{RP}^4$. Thus, additivity of semi-simple crystallizations (Proposition \ref{prop:connected_sum}) gives a huge class of PL $4$-manifolds which admit semi-simple crystallizations.}
\end{remark}

\medskip

The following proposition proves that semi-simple crystallizations are ``minimal'' both with respect to the invariant gem-complexity and with respect to the invariant regular genus.  Moreover,  a lot of details about their combinatorial structure are obtained.

\begin{proposition} \label{theorem 1bis}
Let $(\Gamma,\gamma)$ be a semi-simple crystallization of type $m$. If $M$ denotes the PL $4$-manifold (with $rk(\pi_1(M^4))=m$) represented by $\Gamma,$ then:
$$ \begin{array}{lll}
\mathit{k}(M) &=& 3 \chi (M) + 10m -6;\\
\mathcal G(M)&=& 2 \chi (M) + 5m -4;\\
\mathit{k}(M) &=& \frac{3 \mathcal G(M)  +5m} 2.
\end{array}
$$
\noindent
Moreover:
\begin{enumerate}[(i)]
\item $\rho_\varepsilon (\Gamma) =  \mathcal G(M) = 2 \chi (M) + 5m -4$ \ \ for any cyclic permutation $\varepsilon$ of $\Delta_4,$
\item $\# V(\Gamma) =  2 (\mathit{k}(M) +1)  =  6 \chi (M) + 20m -10, $
\item $g_{ij} = \chi (M) + 4m -1$   \ \ for any pair $i,j \in \Delta_4,$
\item $\rho_{\varepsilon} (\Gamma_{\Delta_4 \setminus \{i\}})  =  \frac{\mathcal G(M) -m}2  =  \chi (M) + 2m -2$ \ \ for any cyclic permutation $\varepsilon$ of $\Delta_4$ and for any color $i\in \Delta_4.$
\end{enumerate}
\end{proposition}

\begin{proof}
Let $(\Gamma,\gamma)$ be a crystallization of $M$, with $\#V(\Gamma)=2p$. From the proof of Theorem \ref{theorem 0}, we have $2p=6 \chi(M)+ 2f_1(\mathcal{K}(\Gamma))-30$. Thus, if $(\Gamma,\gamma)$ is semi-simple, $f_1(\mathcal{K}(\Gamma))=m+1$ and hence $2p= 6 \chi(M)+10(2m-1)$. This proves $\mathit{k}(M) \ = \frac{\#V(\Gamma)}2 -1 \ = \ 3 \chi (M) + 10m -6.$

On the other hand, from Theorem \ref{theorem 0}, we get
$$\mathcal G(M) \geq 2\chi(M) + 5m - 4 = \frac{(\#V(\Gamma)-2)-5m}{3} = \frac{2\mathit{k}(M)-5m}{3}.$$
Now, let $\bar p, q, t_{ijk}$ be as in the proof of Theorem \ref{theorem 0}. In this case $q=0$, $t_{ijk}=0$ for any $i,j,k$ and $\#V(\Gamma)=2 \bar p$; hence
$g_{\varepsilon_i\varepsilon_{i+1}}=\frac{2(m+1)+ \bar p}{3}$ for any cyclic permutation $\varepsilon = (\varepsilon_0,\varepsilon_1, \dots, \varepsilon_4)$.
Therefore,
$\chi_{\varepsilon}(\Gamma)=\sum_{i \in \mathbb{Z}_5}g_{\varepsilon_i\varepsilon_{i+1}}-3 \bar p=\frac{10(m+1)-4 \bar p}{3}$. This implies $\rho_{\varepsilon}(\Gamma) = 1 - \chi_{\varepsilon}(\Gamma)/2 =
\frac{2( \bar p-1)-5m}{3}=\frac{2\mathit{k}(M)-5m}{3}$. Therefore, $\mathcal G(M)=\frac{2\mathit{k}(M)-5m}{3}$. This proves both relation $ \mathcal G(M) \ = \ 2 \chi (M) + 5m -4$ and relation $\mathit{k}(M)= \frac{3 \mathcal G(M)  +5m} 2.$

\smallskip

As pointed out in the proof of Theorem \ref{theorem 0},
$$g_{\varepsilon_i\varepsilon_{i+1}}=\frac{2(m+1)+ \bar p+q}{3} + 2\sum_{0 \leq l <j<k\leq 4}c_{ljk}^{\varepsilon_i\varepsilon_{i+1}}t_{ljk}$$

\noindent where $c_{ljk}^{rs}$ is the element of $A^{-1}$ corresponding to
$\{r,s\}$-row and $\{l,j,k\}$-column of $A^{-1}$. Since in this case both $q$ and all $t_{ljk}$'s are zero, we have $g_{\varepsilon_i\varepsilon_{i+1}}=\frac{2(m+1)+ \bar p}{3}$. Since the same argument holds for any cyclic permutation $\varepsilon$ of $\Delta_4$, relation  $g_{ij} \ = \ \chi (M) + 4m -1$ is proved to be true for any pair $i,j \in \Delta_4.$

Again, from Subsection \ref{sec:genus} we know that, if $\rho$ denotes $\rho_{\varepsilon}(\Gamma)$ and $\rho_{\hat i}$ denotes $\rho_{\varepsilon} (\Gamma_{\Delta_4 \setminus \{i\}})$ then $g_{(i-1)(i+1)}=g_{(i-1)(i)(i+1)} + \rho - \rho_{\hat i}$  and  $ g_{(i-1)(i+1)(i+2)}= 1 + \rho - \rho_{\hat i} - \rho_{\hat {i+3}}$  for any $i \in \mathbb Z_5$.
In the case of a semi-simple crystallization of type $m$, $\rho - \rho_{\hat i} - \rho_{\hat j}=m$ for any pair $i,j \in \Delta_4$. As a consequence,  $\rho_{\hat i}= \frac {\rho - m} 2$ holds for any $i \in \Delta_4$; the proof is completed, since $\rho_{\hat i} = \frac {(2 \chi (M) + 5m -4) - m} 2 = \chi (M) + 2m -2$ directly follows.
\end{proof}

\bigskip

Theorem \ref{theorem 1} is now a direct consequence of Proposition \ref{theorem 1bis}.

\smallskip
\noindent {\em Proof of Theorem} \ref{theorem 1}.
It is sufficient to consider an arbitrary  semi-simple crystallization of $M$, and to apply  Proposition \ref{theorem 1bis}.
\hfill $\Box$

\bigskip

\begin{remark} {\rm If $M$ admits simple crystallizations then, by Theorem \ref{theorem 1}:
$$ \mathit{k}(M) \ = \ 3 \chi (M) -6 = 3 (2 +\beta_2(M)) -6= 3 \beta_2(M)$$
\par \noindent
 and
$$ \mathcal G(M) \ = \ 2 \chi (M) -4 = 2 (2 +\beta_2(M)) -4= 2 \beta_2(M),$$
as already proved in \cite[Theorem 1]{Casali14GemCompl}.}
\end{remark}

\begin{corollary} \label{corollary_free}
Let $M$ be an orientable PL $4$-manifold with $\pi_1(M) = \ast_m \mathbb Z.$ If $M$ admits semi-simple crystallizations, then:
$$ \mathit{k}(M) \ = \ 3 \beta_2 (M) + 4m,$$
$$ \mathcal G(M) \ = \ 2 \beta_2 (M) + m.$$
Moreover, for any semi-simple crystallization $\Gamma$ of $M$:
\begin{enumerate}[(i)]
\item $\rho_\varepsilon (\Gamma) \ = \ 2 \beta_2 (M) + m$ \ \ for any cyclic permutation $\varepsilon$ of $\Delta_4;$
\item $\# V(\Gamma) \ = \ 2 (\mathit{k}(M) +1) \ = \ 6 \beta_2 (M) + 8m +2;$
\item $g_{ij} \ = \ \beta_2 (M) + 2m +1$   \ \ for any pair $i,j \in \Delta_4;$
\item $\rho_{\varepsilon} (\Gamma_{\Delta_4 \setminus \{i\}}) \ = \ \beta_2 (M)$ \ \ for any cyclic permutation $\varepsilon$ of $\Delta_4$ and for any color $i\in \Delta_4.$
\end{enumerate}
\end{corollary}

\begin{proof}
Since $M$ is assumed to be orientable with free fundamental group of rank $m$, \ $\chi (M) = 2 - 2 m + \beta_2(M)$ holds. Hence, all statements follow from the analogous ones in Proposition \ref{theorem 1bis}.
\end{proof}

\medskip

The following proposition gives a generalization of Proposition \ref{gap_genus-rank}(b), within the class of PL 4-manifolds admitting semi-simple crystallizations.

\begin{proposition} \label{classif_PL_genere5}
No PL $4$-manifold $M$ with odd difference $\mathcal G(M) - rk(\pi_1(M))$ admits semi-simple crystallizations. In particular, no simply-connected PL $4$-manifold $M$ with odd regular genus admits simple crystallizations.\end{proposition}

\begin{proof}
It is sufficient to recall that, by Proposition \ref{theorem 1bis}, $\rho_{\varepsilon} (\Gamma_{\Delta_4\setminus \{i\}}) = \frac{\mathcal G(M) -m}2$ \ holds for any semi-simple crystallization $(\Gamma,\gamma)$  of $M$, for any cyclic permutation $\varepsilon$ of $\Delta_4$ and for any color $i\in \Delta_4.$
\end{proof}

Under the assumption of free fundamental group, we have also the following result about the PL classification of orientable PL 4-manifolds admitting semi-simple crystallizations.

\begin{corollary} \label{corollary_beta2}
Let $M$ be an orientable  PL $4$-manifold with $\pi_1(M) = \ast_m \mathbb Z$ and $\beta_2=1$. If $M$ admits semi-simple crystallizations, then $M$ is PL-homeomorphic to  $\mathbb C \mathbb P^2 \#_m (\mathbb S^1 \times \mathbb S^3).$
\end{corollary}

\begin{proof}
In virtue of Corollary \ref{corollary_free}, together with the assumption $\beta_2(M)=1$,  we have  $\mathcal G(M) - rk(\pi_1(M))= 2.$  Hence, the corollary directly follows from Proposition \ref{gap_genus-rank}(c).
\end{proof}

\begin{remark} {\rm As a consequence of the relation between regular genus and gem-complexity for PL 4-manifolds admitting semi-simple crystallizations, it is possible to yield new results about the PL classification via regular genus, within that class of PL 4-manifolds.
For example, if $M$ admits semi-simple crystallizations, and $\mathcal G(M)  = 3,$  $m=1$ hold (resp. $\mathcal G(M)  =  4,$  $m=0$ hold),  then $M$ turns out to be PL-homeomorphic to either $\mathbb C \mathbb P^2 \# (\mathbb S^1 \otimes \mathbb S^3)$ or $\mathbb R \mathbb P^4$   (resp. to either $\mathbb C \mathbb P^2 \# \mathbb C \mathbb P^2$ or $\mathbb S^2 \times \mathbb S^2$ or $\mathbb C \mathbb P^2 \# (- \mathbb C \mathbb P^2)$).
In fact, by Theorem \ref{theorem 1}, $\mathit{k}(M)= \frac{3 \mathcal G(M)  +5m} 2$ holds for any PL 4-manifold admitting semi-simple crystallizations; so, the assumption $\mathcal G(M)  = 3,$  $m=1$  (resp. $\mathcal G(M)  =  4,$  $m=0$)  implies $\mathit{k}(M)=7$ (resp. $\mathit{k}(M)=6$). Hence, the PL-classification of the involved PL 4-manifolds follows from \cite[Proposition 15]{Casali14Cataloguing}.  }
\end{remark}

\bigskip

We conclude the paragraph by deducing the additivity of both regular genus and gem-complexity under connected sum, within the class of PL 4-manifolds admitting semi-simple crystallizations.

\bigskip

\noindent {\em Proof of Theorem} \ref{theorem 2}.
For $1 \leq i \leq 2$, let $(\Gamma_i, \gamma_i)$ be a semi-simple crystallization of the PL $4$-manifold $M_i$ with $rk(\pi_1(M_i))=m_i$. Then, by additivity of semi-simple crystallizations (Proposition \ref{prop:connected_sum}),
$(\Gamma_1 \# \Gamma_2, \gamma_1 \#\gamma_2)$ is a semi-simple crystallization of $M_1 \# M_2$. Since  $rk(\pi_1(M_1 \# M_2))=rk(\pi_1(M_1))+rk(\pi_1(M_2))$ and $\chi(M_1 \# M_2)=\chi(M_1)+\chi(M_2)-2$, Theorem \ref{theorem 2} now follows from Theorem \ref{theorem 1}. \hfill $\Box$

\begin{remark}\label{rematk:add-genus}
{\rm
In \cite[Corollary 6.8]{[GG]}, two classes of closed (not necessarily orientable) $4$-manifolds have been detected, for which additivity of regular genus holds. It has been already pointed out  (see \cite{Casali14GemCompl}) that the first one (characterized by relation $\mathcal G(M)=1 - \frac {\chi(M)} 2 $) consists  of connected sums of $\mathbb S^3$-bundles over $\mathbb S^1,$ while the second one (characterized by relation $\mathcal G(M)= 2 \chi(M) -4$, and consisting of simply-connected PL 4-manifolds, as pointed out in Remark \ref{compare GG}) includes  all PL 4-manifolds admitting simple crystallizations, i.e. semi-simple crystallizations of type zero.
Hence, Theorem \ref{theorem 2} strictly enlarges the set of PL 4-manifolds for which additivity of regular genus is known to hold.}
\end{remark}

\section{Some consequences about regular genus and gem-complexity of product 4-manifolds}\label{sec:consequences}

Theorem \ref{theorem 0} enables to significantly improve some lower bounds for the regular genus of PL 4-manifolds, which have been proved by various authors via different techniques. Meanwhile, similar lower  bounds are obtained also for gem-complexity.

\begin{proposition} \label{M3xS1}
For any $3$-manifold $M$ such that $\pi_1(M)$ is a finitely generated abelian group, we have:
 $$ \mathcal G(M \times \mathbb S^{1}) \ge 5 rk(\pi_1(M)) +1 \ \ \quad and \quad \ \   \mathit{k}(M \times \mathbb S^{1}) \ge 10 rk(\pi_1(M)) +4.$$
In particular,
$$ \mathcal G(L(p,q) \times \mathbb S^{1}) \ge 6  \ \ \quad and \quad \ \   \mathit{k}(L(p,q) \times \mathbb S^{1}) \ge 14.$$
\end{proposition}

\begin{proof}
It is well-known that $ \chi(M)=0$ for any 3-manifold $M$ and $\chi(P \times  Q)=\chi(P) \cdot \chi(Q)$ for any pair $P, Q$ of polyhedra. Moreover, by the fundamental theorem of finitely generated abelian groups,  $ rk(\pi_1(M \times \mathbb S^{1}))=rk(\pi_1(M)) + 1$  holds for any 3-manifold $M$ such that $\pi_1(M)$ is
a finitely generated abelian group.  The statements are now direct consequences of the inequalities proved in  Theorem \ref{theorem 0}.
\end{proof}

\begin{remark} {\rm The statement $ \mathcal G(L(p,q) \times \mathbb S^{1}) \ge 6$ already appears
in \cite{[S]}, by making use of  long  calculations performed in \cite{[CavM]}.
On the other hand, the genus six crystallization of $L(2,1) \times \mathbb S^{1}$ produced in  \cite{[S]} allows to prove the equality $\mathcal G(L(2,1) \times \mathbb S^{1})=6.$ The inequality  $ \mathcal G(L(p,q) \times \mathbb S^{1}) \le 6(p-1)$  is also proved in \cite{[S]}.}
\end{remark}

\begin{proposition} \label{prodotto superfici}
Let $T_g$ (resp. $U_h$) denote  the orientable (resp. non-orientable) surface of genus $g\ge 0$ (resp. $h \ge 1$). Then,
$$ \mathcal G(T_g \times T_r) \ge 8 gr + 2 g + 2 r + 4 \ \ \quad and \quad \ \   \mathit{k}(T_g \times T_r) \ge 12 gr + 8 g + 8 r + 6,$$
$$ \mathcal G(T_g \times U_h) \ge 4 gh + 2g + h + 4 \ \ \quad and \quad \ \   \mathit{k}(T_g \times U_h) \ge 6 gh + 8 g + 4 h + 6,$$
$$ \mathcal G(U_h \times U_k) \ge 2 hk + h + k + 4 \ \ \ \quad and \quad \ \   \mathit{k}(U_h \times U_k) \ge 3 hk + 4 h + 4 k + 6.$$
In particular,
$$ \mathcal G(\mathbb S^2 \times T_g) \ge 2g+4 \ \ \quad and \quad \ \   \mathit{k}(\mathbb S^2 \times T_g) \ge  8 g + 6,$$
$$ \mathcal G(\mathbb S^2 \times U_h) \ge h+4 \ \ \ \quad and \quad \ \   \mathit{k}(\mathbb S^2 \times U_h) \ge  4 h + 6.$$
\end{proposition}

\begin{proof}
The following facts are well-known: $ \chi(T_g)=2-2g,$ $ \chi(U_h)=2-h,$ $rk(\pi_1(T_g))= 2g$ and $rk(\pi_1(U_h))= h.$ Moreover, $ \chi(P \times  Q)=\chi(P) \cdot \chi(Q)$ for any pair $P, Q$ of polyhedra.
On the other hand, it is not difficult to prove that $rk(\pi_1(P \times Q))=rk(\pi_1(P)) + rk(\pi_1(Q))$ for any pair $P, Q$ of surfaces: in fact, $rk(G \times H) \leq rk(G) + rk(H)$ holds for any pair $G, H$ of groups, while the equality for the fundamental groups of surfaces is a consequence of the equality regarding the first Betti numbers (with integer coefficients in the orientable case and with $\mathbb Z_2$ coefficients in the non-orientable case).
The statements are now direct consequences of the inequalities proved in Theorem \ref{theorem 0}.
\end{proof}

\begin{remark} {\rm The inequalities concerning regular genus in Proposition \ref{prodotto superfici} strictly improve the similar ones obtained in   \cite[Corollary 6.6]{[GG]}.}
\end{remark}

\medskip

\begin{figure}[ht]
\tikzstyle{vert}=[circle, draw, fill=black!100, inner sep=0pt, minimum width=4pt]
\tikzstyle{vertex}=[circle, draw, fill=black!00, inner sep=0pt, minimum width=4pt]
\tikzstyle{ver}=[]
\tikzstyle{extra}=[circle, draw, fill=black!50, inner sep=0pt, minimum width=2pt]
\tikzstyle{edge} = [draw,thick,-]
\centering

\begin{tikzpicture}[scale=0.7]

\begin{scope}[shift={(4,0)}]
\foreach \x/\y in {0/20,120/1,240/2}{
\node[vertex] (\y) at (\x:2){};
}
\end{scope}

\begin{scope}[shift={(4,0)}]
\foreach \x/\y in {60/7,180/6,300/9}{
\node[vertex] (\y) at (\x:2){ };
}
\end{scope}

\begin{scope}[shift={(-4,0)}]
\foreach \x/\y in {300/8,60/10,180/14}{
\node[vertex] (\y) at (\x:2){ };
}
\end{scope}

\begin{scope}[shift={(-4,0)}]
\foreach \x/\y in {0/18,120/22,240/21}{
\node[vertex] (\y) at (\x:2){ };
}
\end{scope}

\begin{scope}[shift={(0,4)}]
\node[vertex] (0) at (1,-1){};
\node[vertex] (4) at (1,1){};
\node[vertex] (11) at (-1,1){};
\node[vertex] (3) at (-1,-1){};
\end{scope}

\begin{scope}[shift={(4,-6)}]
\node[vertex] (13) at (1,-1){};
\node[vertex] (17) at (1,1){};
\node[vertex] (5) at (-1,1){};
\node[vertex] (16) at (-1,-1){};
\end{scope}

\begin{scope}[shift={(-4,-6)}]
\node[vertex] (23) at (1,-1){};
\node[vertex] (15) at (1,1){};
\node[vertex] (19) at (-1,1){};
\node[vertex] (12) at (-1,-1){};
\end{scope}

\foreach \x/\y in {1/7,20/9,2/6,8/18,10/22,14/21,0/4,11/3,5/17,13/16,12/19,15/23}{
\draw[line width=2pt, line cap=rectengle, dash pattern=on 1pt off 1] (\x) -- (\y);
}
\foreach \x/\y in {7/20,9/2,1/6,18/10,22/14,21/8,4/11,3/0,17/13,5/16,19/15,23/12}{
\draw [line width=3pt, line cap=round, dash pattern=on 0pt off 2\pgflinewidth] 	(\x) -- (\y);}

\foreach \x/\y in {1/7,20/9,2/6,8/18,10/22,14/21,0/4,11/3,5/17,13/16,12/19,15/23}{
\path[edge](\x) -- (\y);}
\foreach \x/\y in {7/20,9/2,1/6,18/10,22/14,21/8,4/11,3/0,17/13,5/16,19/15,23/12}{
\path[edge](\x) -- (\y);}

\foreach \x/\y in {0/1,8/2,4/7,3/10,5/15,6/18,22/11,23/16}{
\path[edge](\x) -- (\y);}
\draw[edge] plot [smooth,tension=1] coordinates{(14)(-4,-3)(4,-3)(20)};
\draw[edge] plot [smooth,tension=1] coordinates{(21)(-3,-2.5)(3,-2.5)(9)};
\draw[edge] plot [smooth,tension=1] coordinates{(12)(-5.5,-6)(19)};
\draw[edge] plot [smooth,tension=1] coordinates{(13)(5.5,-6)(17)};
\draw[edge, dashed] plot [smooth,tension=1] coordinates{(14)(-6,1)(22)};
\draw[edge, dashed] plot [smooth,tension=1] coordinates{(1)(1,-1)(5)};
\draw[edge, dashed] plot [smooth,tension=1] coordinates{(7)(7,-1)(17)};
\draw[edge, dashed] plot [smooth,tension=1] coordinates{(0)(1,0)(2)};
\foreach \x/\y in {4/13,3/9,6/16,8/15,10/12,11/20,18/23,19/21}{
\path[edge, dashed](\x) -- (\y);}

\foreach \x/\y in {4/12,7/19,16/22,17/21}{
\path[edge, dotted](\x) -- (\y);}
\draw[edge, dotted] plot [smooth,tension=1] coordinates{(0)(1.8,2)(1)};
\draw[edge, dotted] plot [smooth,tension=1] coordinates{(2)(2.6,-0.7)(6)};
\draw[edge, dotted] plot [smooth,tension=1] coordinates{(3)(2.5,1)(9)};
\draw[edge, dotted] plot [smooth,tension=1] coordinates{(15)(3,-2.5)(20)};
\draw[edge, dotted] plot [smooth,tension=1] coordinates{(11)(0,0)(23)};
\draw[edge, dotted] plot [smooth,tension=1] coordinates{(8)(-3,-0.5)(18)};
\draw[edge, dotted] plot [smooth,tension=1] coordinates{(5)(0,-2.5)(14)};
\draw[edge, dotted] plot [smooth,tension=1] coordinates{(10)(3,-4)(13)};
\end{tikzpicture}
\caption{A crystallization of $ \mathbb S^2 \times \mathbb{RP}^2$ (with genus five and order 24)}
\label{fig:RP2S2}
\end{figure}
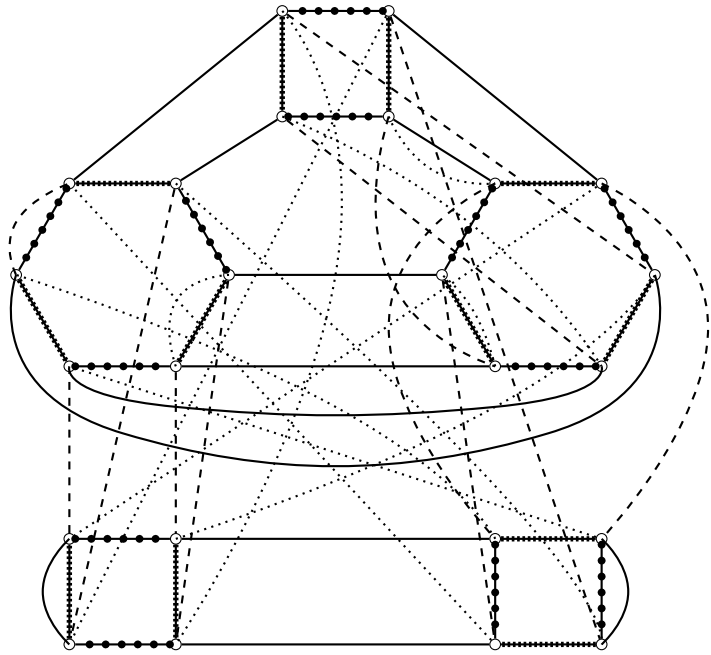

Finally, for $h=1$, the last inequality of Proposition  \ref{prodotto superfici} concerning regular genus (resp. gem-complexity), together with the existence of the genus five (resp. order $24$) crystallization of $\mathbb S^2 \times \mathbb R \mathbb P^2$ depicted in Figure 3, allows the exact calculation of the regular genus (resp. an estimation with ``strict range'' of the gem-complexity) of the involved PL 4-manifold.

\begin{proposition} \label{piano proiettivo per sfera}
$$ \mathcal G(\mathbb S^2 \times \mathbb R \mathbb P^2) =5 \ \ \quad and \quad \ \  \mathit{k}(\mathbb S^2 \times \mathbb R \mathbb P^2) \in \{10,11\}.$$
\end{proposition}
\vskip -0.8pc
\hfill $\Box$

\bigskip

\noindent {\bf Acknowledgements:}
 The authors express their gratitude to Prof. Basudeb Datta and Dr. Jonathan Spreer for helpful comments.
The first author is supported by CSIR, India for SPM Fellowship and the UGC Centre for Advanced Studies. The second author is supported by the ``National Group for Algebraic and Geometric Structures, and their Applications'' (GNSAGA - INDAM) and by M.I.U.R. of Italy (project ``Strutture Geometriche, Combinatoria e loro Applicazioni'').

\end{document}